\documentclass{amsart}[12pt]
\usepackage{amsmath, amsthm, amscd, amsfonts,}
\usepackage{graphicx,float}

\usepackage{amssymb}

\usepackage{pb-diagram}

\newcommand{\PP}{{\mathbb{P}}}

\newcommand{\url}[1]{{\tt #1}}

\DeclareMathOperator{\range}{range}

\DeclareMathOperator{\len}{lh}

\DeclareMathOperator{\dom}{dom}

\DeclareMathOperator{\Col}{Col}
\DeclareMathOperator{\Add}{Add}

\DeclareMathOperator{\Suc}{Suc}
\DeclareMathOperator{\Lev}{Lev}

\def\MPB{{\mathbb{P}}}
\def\MQB{{\mathbb{Q}}}

\newtheorem{theorem}{Theorem}[section]
\newtheorem{lemma}[theorem]{Lemma}

\newtheorem{definition}[theorem]{Definition}
\newtheorem{open Question}[theorem]{Open Question}

\newtheorem{notation}[theorem]{Notation}
\newtheorem{claim}[theorem]{Claim}
\numberwithin{equation}{section}

\def\MPB{{\mathbb{P}}}
\def\MQB{{\mathbb{Q}}}

\def\rmark{\mbox{$\rm\bf\rule{0.06em}{1.45ex}\kern-0.05em R$}}
\def\pmark{\mbox{$\rm\bf\rule{0.06em}{1.45ex}\kern-0.05em P$}}
\def\nmark{\mbox{$\rm\bf\rule{0.06em}{1.45ex}\kern-0.05em N$}}
\def\vdash{\mbox{$\rm\| \kern-0.13em -$}}

\usepackage{pb-diagram}

\def\rmark{\mbox{$\rm\bf\rule{0.06em}{1.45ex}\kern-0.05em R$}}
\def\pmark{\mbox{$\rm\bf\rule{0.06em}{1.45ex}\kern-0.05em P$}}
\def\nmark{\mbox{$\rm\bf\rule{0.06em}{1.45ex}\kern-0.05em N$}}
\def\vdash{\mbox{$\rm\| \kern-0.13em -$}}

\begin{document}

\title[Strongly compct diagonal Prikry forcing]{Strongly compct diagonal Prikry forcing}

\author[M. Golshani ]{Mohammad Golshani}

\thanks{The first author's research was in part supported by a grant from IPM (No. 98030417).}
\maketitle

\begin{abstract}
We define a version of Gitik-Sharon diagonal Prikry forcing using a strongly compact cardinal, and prove its basic properties.
\end{abstract}
\maketitle

\section{Introduction}
In \cite{gitik-sharon}, Gitik and Sharon introduced a new forcing notion,  diagonal (supercompact) Prikry forcing, to answer some questions of Cummings, Foreman, Magidor and Woodin. So starting from a supercompact cardinal $\kappa$, they introduced a generic extension in which the following hold:
\begin{enumerate}
\item $\kappa$ is a singular limit cardinal of cofinality $\omega$ and $2^\kappa > \kappa^+,$
\item There exists a very good scale at $\kappa,$
\item There is a bad scale at $\kappa.$
\end{enumerate}
In this paper we define a strongly compact version
of Gitik-Sharon forcing that we call \emph{strongly compact diagonal Prikry forcing}, prove its basic properties and show that it shares  all properties of diagonal Prikry forcing.

\section{Strongly compact diagonal Prikry forcing}
In this section we define our \emph{strongly compact diagonal Prikry forcing}. Assume $\kappa$
is a strongly compact cardinal, and let
\[
\kappa=\kappa_0 < \kappa_1 < \dots < \kappa_n < \dots
\]
be an increasing sequence of regular cardinals with limit $\kappa_\omega.$ Let $U$ be a fine measure on $P_{\kappa}(\kappa_\omega^+),$
and for each $n< \omega$ let $U_n$ be its projection to $P_{\kappa}(\kappa_n):$
\[
X \in U_n \Leftrightarrow X \subseteq P_{\kappa}(\kappa_n) \wedge \{ P \in P_{\kappa}(\kappa_\omega^+): P \cap \kappa_n \in X   \} \in U.
\]
Let
\begin{center}
$K_n=\{ P \in P_{\kappa}(\kappa_n): P \cap \kappa$ is inaccessible $\}$.
\end{center}
Then $K_n \in U_n.$
Corresponding to the sequences $\bar{\kappa}=\langle \kappa_0, \dots, \kappa_n, \dots        \rangle$ and $\bar{U}= \langle U_0, \dots, U_n, \dots    \rangle$ we define the forcing notion
$\MPB=\MPB_{\bar{\kappa}, \bar{U}}$ as follows.
\begin{definition}
A condition in $\MPB$ is a finite sequence
\[
p= \langle  P_0, \dots, P_{n-1}, T    \rangle
\]
where:
\begin{enumerate}
\item For $i < n, P_i \in K_i,$

\item $P_0 \prec P_1 \prec \dots \prec P_{n-1},$ where
\[
P \prec Q \Leftrightarrow otp(P)=\lambda_P < \kappa_Q= Q \cap \kappa,
\]
\item $T$ is a $\bar{U}$-tree with trunk $\langle P_0, \dots, P_{n-1}         \rangle,$ which means:
\begin{enumerate}
\item $T$ is a tree, whose nodes are finite sequences $\langle  Q_0, \dots, Q_{m-1} \rangle$, such that each $Q_i \in K_i$ and
$Q_0 \prec Q_1 \prec \dots \prec Q_{m-1},$ ordered by end extension,

\item The trunk of $T$ is $t=\langle P_0, \dots, P_{n-1}         \rangle,$ which means $t \in T$ and for any $s \in T, s \unlhd t$ or $t \unlhd s,$

\item If $s=\langle  Q_0, \dots, Q_{m-1} \rangle \unrhd t,$ then
\[
\Suc_T(s)=\{ Q \in K_{m}: s^{\frown} \langle Q  \rangle \in T          \} \in U_{m}.
\]
\end{enumerate}
\end{enumerate}
\end{definition}
Given a condition $p \in \MPB,$ we denote it by
\[
p=\langle P_0^p, \dots, P_{\len(p)-1}^p, T^p             \rangle
\]
and call $\len(p)$ the length of $p$. We allow $\len(p)=0,$ which just means $p$ has no $P$'s in its definition. We also call $\langle P_0^p, \dots, P_{\len(p)-1}^p          \rangle$ the lower part of $p$.
\begin{definition}
Let $T$ be a tree as above and $ s \in T.$ Then
\[
T_s = \{ u\in T: u \unlhd s \text{~or~} s \unlhd u           \}.
\]
\end{definition}
\begin{definition}
Let $p, q \in \MPB.$ Then $p \leq q$ iff
\begin{enumerate}
\item $\len(p) \geq \len(q),$

\item For all $i < \len(q), P_i^p=P_i^q,$

\item For all $\len(q) \leq i < \len(p), P_i^p \in \Suc_{T^q}(\langle P_0^p, \dots P_{i-1}^p     \rangle)$,

\item $T^p \subseteq T^q_{\langle P_0^p, \dots, P_{\len(p)-1}^p     \rangle}$.
\end{enumerate}
\end{definition}
\begin{definition}
Let $p, q \in \MPB.$ We say $p$ is a Prikry or a direct extension of $q, p \leq^* q,$ iff $p \leq q$ and $\len(p)=\len(q).$
\end{definition}
Before we continue, let us introduce a notation that will become useful later.
\begin{notation}
Let $\Xi$ be the tree of possible lower parts:
\begin{center}
$\Xi = \{\langle P_0, \dots, P_{n-1}   \rangle: n< \omega, P_i \in K_i, P_0 \prec \dots P_{n-1}\}$.
\end{center}
Also we denote each $t \in \Xi$ as $t= \langle  P_0^t, \dots P_{\len(t)-1}^t      \rangle$.
\end{notation}

We now study the basic properties of the forcing notion $(\MPB, \leq, \leq^*)$.
\begin{lemma}
$(\MPB, \leq)$ satisfies the $\kappa_\omega^+$-c.c.
\end{lemma}
\begin{proof}
This follows easily using the fact that if $p$ and $q$ have the same lower part, then they are compatible, and that
\[
|\{\langle P_0^p, \dots, P_{\len(p)-1}^p \rangle : p\in \MPB    \}| \leq\kappa_\omega.
\]
\end{proof}
\begin{lemma}
$(\MPB, \leq^*)$ is $\kappa$-closed.
\end{lemma}
\begin{proof}
By the $\kappa$-completeness of $U_n$'s.
\end{proof}
We now show that $(\MPB, \leq, \leq^*)$ is a Prikry type forcing notion.
\begin{lemma}
$(\MPB, \leq, \leq^*)$ satisfies the Prikry property.
\end{lemma}
\begin{proof}
Let $p \in \MPB$ and let $\sigma$ be a statement of the forcing language $(\MPB, \leq).$ We find $q \leq^* p$
which decides $\sigma.$ Assume this is not true.

Call a lower part $t=\langle P_0, \dots, P_{n-1}    \rangle$
indecisive if there is no tree $T$ with trunk $t$ such that $p= \langle P_0, \dots, P_{n-1}, T  \rangle \in \MPB$
and $p$ decides $\sigma.$ Otherwise $t$ is called decisive. Note that by our assumption
the lower part of $p$ is indecisive.
\begin{claim}
If $t=\langle P_0, \dots, P_{n-1}    \rangle$ is indecisive, then
\[
\{ P \in K_{n}: t^{\frown} \langle P \rangle \text{~is indecisive}\} \in U_{n}.
\]
\end{claim}
\begin{proof}
Assume otherwise, so
\[
X=\{ P \in K_{n}: t^{\frown} \langle P \rangle \text{~is decisive}\} \in U_{n}.
\]
For $P \in X$ pick a tree $T_P$ and $i<2$ such that $q_P= \langle t^{\frown} \langle P \rangle, T_P \rangle \in \MPB$
and $q_P \Vdash~ ^{i}\sigma$ (where $^{0}\sigma=\sigma$ and $^{1}\sigma=\neg \sigma$).
Let $i<2$ be such that
\[
Y=\{ P \in X: q_P  \Vdash~ ^{i}\sigma \} \in U_{n}.
\]
Let $T$ be a tree with trunk $t$, so that $\Suc_{T}(s)=Y,$ and for each $P \in Y, T_{\langle   t^{\frown} \langle P \rangle   \rangle}=T_P.$
Let $p= \langle t, T \rangle$. Then $p \in \MPB,$ and any extension of $p$ extends some $q_P, P \in Y.$ It follows that $p\Vdash ^{i}\sigma,$
hence $t$ is decisive, a contradiction.
\end{proof}
By the above claim and by induction, we can find a tree $T$ with trunk $\langle    P_0^p, \dots, P_{\len(p)-1}^p      \rangle$
such that all nodes $t \in T, t \unrhd \langle    P_0^p, \dots, P_{\len(p)-1}^p      \rangle$
are indecisive. Let $q= \langle     P_0^p, \dots, P_{\len(p)-1}^p, T   \rangle$.
Let $r \leq q$ and $r$ decides $\sigma.$ Then $\langle P_0^r, \dots, P_{\len(r)-1}^r    \rangle \in T$
and it is decisive, a contradiction. The lemma follows.
\end{proof}
Let $G$ be $\MPB$-generic over $V$, and let
$\langle  P_i: i<\omega     \rangle$
be the Prikry sequence added by $G$, where
$P_i=P_i^p,$ for some (and hence all) $p\in G$ wit $\len(p) > i.$ Then
\[
P_0 \prec P_1 \prec \dots \prec P_i \prec \dots.
\]
\begin{lemma}
For any $n \leq \omega,$
\[
\kappa_n = \bigcup \{P_i \cap \kappa_n: i< \omega \},
\]
in particular all cardinals in $(\kappa, \kappa_\omega)$ are collapsed into $\kappa.$
\end{lemma}
Let us summarize the properties of forcing notion $\MPB.$
\begin{theorem}
Let $G$ be $\MPB$-generic over $V$. Then

$(a)$ $cf^{V[G]}(\kappa)=\omega,$

$(b)$ $\kappa^{+V[G]}=\kappa_\omega^+,$

$(c)$ No bounded subsets of $\kappa$ are added, in particular all cardinals $\leq \kappa$ are preserved.
\end{theorem}

\section{More on strongly compact diagonal Prikry forcing}
In this section we prove some more properties of the forcing notion $\MPB$ introduced in the previous section.
Let $G$ be $\MPB$-generic over $V$, and let $\langle P_i: i<\omega   \rangle$
be the corresponding Prikry generic sequence. It is easily seen that
\[
G= \{p \in \MPB: \langle P_0^p, \dots, P_{\len(p)-1}^p=  \langle P_0, \dots, P_{\len(p)-1}  \rangle \text{~and~} \forall i \geq \len(p), P_i \in \Suc_{T^p}(\langle  P_0, \dots, P_{i-1} \rangle)       \},
\]
hence $V[G]=V[\langle P_i: i<\omega   \rangle]$.
\begin{lemma}
(Diagonal intersection lemma)
For each $t \in \Xi$, let $T^t$ be a $\bar{U}$-tree with trunk $t$ such that $\langle  t, T^t \rangle \in \MPB.$ Then there is a $\bar{U}$-tree $S$ with trunk $\langle \rangle,$
so that for each $t\in S$, $\langle t, S_t\rangle \leq \langle t, T^t \rangle$.
\end{lemma}
\begin{proof}
Define the tree $S$ by induction on levels so that for each $t \in S,$
\begin{center}
$\Suc_S(t)=\bigcap_{i \leq \len(t)} \Suc_{T^{t \upharpoonright i}}(t) \in K_{\len(t)}$.
\end{center}
We show that $S$ is as required. Thus let $t\in S.$ We need to show that $\langle t, S_t\rangle \leq \langle t, T^t \rangle$, i.e., $S_t \subseteq T^t.$
Thus assume $t \unlhd s \in S.$ Then
\[
s \in \Suc_S(s \upharpoonright \len(s)-1) \subseteq \Suc_{T^t}(s \upharpoonright \len(s)-1),
\]
so $s \in T^t.$
\end{proof}

\begin{lemma}
Assume $A \in V[G]$ is a set of ordinals of order type $\beta,$ where $\omega < \beta = cf^V(\beta) < \kappa.$
Then there exists an unbounded $B \subseteq A$ with $B \in V.$
\end{lemma}
\begin{proof}
For each $p\in G$ set $A_p= \{\alpha: p \Vdash~ \alpha \in \dot{A}\}.$ Then $A= \bigcup_{p\in G}A_p.$ Note that
in $V[G]$, $cf(\beta)=\beta > \omega,$ so for some $n< \omega,$ the set $A'=\bigcup_{p\in G, \len(p)=n}A_p$
is an unbounded subset of $A$.

Let $f\in V[G], f:\beta \rightarrow A'$ enumerate $A'$. For each $\alpha< \beta$ let $p_\alpha=\langle  P_0, \dots, P_{n-1}, T^\alpha     \rangle \in \MPB$
be such that $p_\alpha$ decides $\dot{f}(\alpha),$ where $\langle P_0, \dots, P_i, \dots  \rangle$
is the generic Prikry sequence. Let $p$ be  such that the lower part of $p$ is $\langle P_0, \dots, P_{n-1}   \rangle$ and for each $\langle P_0, \dots, P_{n-1}   \rangle \unlhd t\in T^p,$
 $\Suc_{T^p}(t)= \bigcap_{\alpha < \beta}\Suc_{T^\alpha}(t).$

Then $p\in \MPB$ and $p$ decides $\dot{f}.$ The result follows immediately.
\end{proof}
\begin{lemma}
(Bounding lemma) Assume $\forall n<\omega, \kappa_n=\kappa^{+n}$ (recall $\langle \kappa_n: n<\omega  \rangle$ is the sequence we fixed at the beginning). Let $\eta: \omega \rightarrow \kappa$ be such that $\eta(n)>n$ is a successor ordinal. Let
$\langle P_i: i<\omega   \rangle$
be the Prikry generic sequence,  and let $h \in V[\langle P_i: i<\omega   \rangle]$ with $h \in \prod_{i<\omega} \kappa_{P_i}^{+\eta(i)}$.
Then there exists $\langle H_i: i < \omega \rangle \in V$, so that:
\begin{enumerate}
\item For each $i$,
$\dom(H_i)= K_{i},$

\item  For all  $Q \in \dom(H_i)$, $H_i(Q) < \kappa_Q^{+\eta(i)}$,

\item For all large $i$,
$h(i) < H_{i}(P_i)$.
\end{enumerate}
\end{lemma}
\begin{proof}
Assume for simplicity that the trivial condition forces $\dot{h}$ is as in the statement of the lemma.
For any $t\in \Xi,$  by the Prikry property, let $q_t = \langle t, H^t \rangle \in \MPB$ be such that $q_t$ decides $\dot{h}(\len(t)-1)$, say
 $q_t \Vdash~ \dot{h}(\len(t)-1)= g(t) < \kappa_{P^t_{\len(t)-1}}^{+\eta(\len(t)-1)}.$

By diagonal intersection lemma, we can find a tree $S$ so that for each $t\in S, \langle t, S_t \rangle \leq q_t.$ Let
$p=\langle  \langle \rangle, S \rangle.$ Then for any $i<\omega,$
 \[
 p \Vdash~ \dot{h}(i)=g(\langle  P_0, \dots, P_i     \rangle).
 \]

For any $i<\omega$
let $\dom(H_i)= K_{i},$ and for $Q \in K_i$ set
\[
H_i(Q) = \sup\{g(t):   t \in \Xi, \len(t)=i+1, P^t_{i}=Q    \}+1.
\]
By a simple counting argument, $H_i(Q) \leq \kappa_Q^{+i} < \kappa_Q^{+\eta(i)}$.

\end{proof}

School of Mathematics, Institute for Research in Fundamental Sciences (IPM), P.O. Box:
19395-5746, Tehran-Iran.

E-mail address: golshani.m@gmail.com

\end{document}